\documentclass[12pt]{article}
\usepackage{amssymb,latexsym}
\usepackage{enumerate}
\usepackage{amsmath}
\usepackage{amsthm}
\usepackage{amsfonts}
\usepackage{color}

\newtheorem{theorem}{Theorem}[section]
\newtheorem{Corol}[theorem]{Corollary}
\newtheorem{Lemma}[theorem]{Lemma}
\newtheorem{Prop}[theorem]{Proposition}

\theoremstyle{definition}

\newtheorem{Remark}[theorem]{Remark}
\newtheorem{Example}[theorem]{Example}

\numberwithin{equation}{section}

\newcommand{\tr}{\mathrm{tr} }

\begin{document}


\baselineskip=17pt



\title{Space-like Surfaces in  Minkowski Space $\mathbb E^4_1$  with Pointwise 1-Type Gauss Map
}
\author{U\u gur Dursun\footnote
{Adress: Istanbul Technical University, Faculty of Science and Letters, Department of Mathematics
34469, Maslak, Istanbul, TURKEY
} \footnote{e-mail: udursun@itu.edu.tr} and Nurettin Cenk Turgay \footnotemark[1] \footnote{e-mail: turgayn@itu.edu.tr}}

\date{}

\maketitle
\begin{abstract}
In this work we firstly classify space-like surfaces in  Minkowski space $\mathbb E^4_1$, 
de-Sitter space $\mathbb S^3_1$ and hyperbolic space $\mathbb H^3$
with harmonic Gauss map. 
Then we give a characterization and classification  of space-like surfaces with 
pointwise 1-type Gauss map of the first kind. We also give some explicit examples.


\end{abstract}

\maketitle

\section{Introduction}

In late 1970's B. Y. Chen introduced the notion of finite type submanifolds of Euclidean space
 \cite{CH1}. Since then many works have been done to characterize or classify submanifolds of 
 Euclidean space or pseudo-Euclidean space in terms of finite type.
Also,  B. Y. Chen and P. Piccinni extended the notion 
of finite type to differentiable maps, in particular, to Gauss map of submanifolds in \cite{Chen-Piccinni}. 
A smooth map $\phi$ on  a submanifold $M$ of a Euclidean space or 
a pseudo-Euclidean space is said to be of {\it finite type} if  $\phi$ can be expressed  as a finite  
sum of eigenfunctions 
of  the Laplacian $\Delta$ of $M$, that is, 
$\phi = \phi_0 + \sum_{i=1}^k \phi_i$, where $\phi_0$ is a constant map, $\phi_1, \dots, \phi_k$ 
non-constant maps such that $\Delta \phi_i = \lambda_i \phi_i, \; \lambda_i \in \mathbb R$, $i=1, \dots, k$.

If a submanifold $M$  of a Euclidean space or a pseudo-Euclidean space  has  1-type Gauss map
$\nu$, then  $\nu$ satisfies $ \Delta \nu  = \lambda (\nu+C)$ for some $\lambda \in \mathbb
R$ and  some constant vector $C$.  In \cite{Chen-Piccinni}, B. Y. Chen and P. Piccinni studied compact submanifolds 
of Euclidean spaces with finite type Gauss map. Several articles also appeared  on   submanifolds with finite type Gauss map 
 (cf. \cite{Baik,Baik-Blair,Baik-Chen-Verst,Baik-Verst,Yoon,Yoon-2}).

However, the Laplacian of  the Gauss map of several surfaces and hypersurfaces such as  
helicoids  of  the  1st, 2nd, and 3rd  kind, conjugate  Enneper's  surface of the second kind and 
 B-scrolls  in  a 3-dimensional Minkowski space $\mathbb E^3_1$,  
 generalized  catenoids, spherical n-cones, hyperbolical n-cones and 
 Enneper's hypersurfaces in $\mathbb E^{n+1}_1$   take the form
\begin{equation}\label{PW1TypeDefinition}
 \Delta \nu =f(\nu +C)
\end{equation}
for some smooth function $f$ on $M$ and some constant vector
$C$ (\cite{UDur2,Kim-Yoon}). A submanifold   of a pseudo-Euclidean space is said to have
{\it pointwise 1-type Gauss map} if its  Gauss map satisfies
\eqref{PW1TypeDefinition} for some smooth  function $f$ on $M$  and some
constant vector $C$.   In particular, if $C$ is zero, it is 
said to be of {\it the first kind}.
Otherwise, it  is said to be of {\it the second kind} 
(cf. \cite{Arslan-and,CCK,Choi-Kim,UDur,UDur3,KKKM,Kim-Yoon-2}).

\begin{Remark} \label{RemarkDefinition} 
The Gauss map $\nu$ of  a totally geodesic submanifold $M$ in $\mathbb E^m_1$ is a constant vector and 
$\Delta \nu = 0$, i.e., it is harmonic. For $f=0$ if we write $\Delta \nu  = 0 \cdot \nu$, 
then  $M$ has pointwise 1-type Gauss map of the first kind. 
If we choose    $C=-\nu$, then \eqref{PW1TypeDefinition} holds for any non-zero  smooth function $f$. 
In this case  $M$ has pointwise 1-type Gauss map of the second kind.
Therefore, a totally geodesic submanifold  in $\mathbb E^m_1$ is a trivial submanifold with 
pointwise 1-type Gauss map of both  the first kind and the second kind.  
\end{Remark}

The complete classification of 
ruled surfaces in $\mathbb E^3_1$ with pointwise  1-type Gauss map of 
the first kind was obtained in \cite{Kim-Yoon}. 
Also, a complete classification of rational surfaces of revolution 
in $\mathbb E^3_1$ satisfying \eqref{PW1TypeDefinition} was recently given 
in \cite{KKKM}, and it was proved that  a right circular cone 
and a  hyperbolic cone in $\mathbb E^3_1$ are the only rational surfaces of 
revolution  in $\mathbb E^3_1$ with pointwise  1-type Gauss map of the second kind.  
The first   author studied  rotational hypersurfaces in Lorentz-Minkowski space  
with pointwise  1-type Gauss map  \cite{UDur2},  
Moreover, in \cite{Kim-Yoon-3} a complete classification of cylinderical and 
non-cylinderical surfaces in $\mathbb E^m_1$ with pointwise 1-type Gauss map 
of the first kind was obtained.

In this article, we study space-like surfaces in $\mathbb E^4_1$ 
with pointwise  1-type Gauss map of the first kind. 
Surfaces with harmonic Gauss map  in $\mathbb E^4_1$ are of global 1-type Gauss 
map of the first kind. 
We first give a characterization and classification of maximal surfaces and 
non-maximal space-like surfaces in $\mathbb E^4_1$ with harmonic Gauss map.
We also prove that  oriented maximal surfaces and surfaces with light-like mean curvature vector
in $\mathbb E^4_1$  with harmonic 
Gauss map  are the only surfaces  in $\mathbb E^4_1$  with (global) 1-type 
Gauss map of the first kind. 

Then, we obtain  the necessary and sufficient conditions on non-maximal space-like surfaces 
in $\mathbb E^4_1$  with pointwise 1-type Gauss map of the first kind and 
we give a classification of such surfaces.
Further, we prove that an oriented non-maximal space-like surface in $\mathbb E^4_1$ 
has (global) 1-type Gauss map of  the first kind if and only if the surface 
has constant Gaussian curvature and parallel mean curvature. 

\section{Prelimineries}

Let $\mathbb E^m_t$ denote the pseudo-Euclidean $m$-space with the canonical 
pseudo-Euclidean metric tensor of index $t$ given by  
$$
 g=-\sum\limits_{i=1}^t dx_i^2+\sum\limits_{j=t+1}^m dx_j^2,
$$
where $(x_1, x_2, \hdots, x_m)$  is a rectangular coordinate system in $\mathbb E^m_t$.
We put 
\begin{eqnarray} 
\mathbb S^{m-1}_t(r^2)&=&\{x\in\mathbb E^m_t: \langle x, x \rangle=r^{-2}\},\notag
\\  
\mathbb H^{m-1}_{t-1}(-r^2)&=&\{x\in\mathbb E^m_t: \langle x, x \rangle=-r^{-2}\},\notag
\end{eqnarray}
where $\langle\ ,\ \rangle$ is the indefinite inner product of $\mathbb E^m_t$.
Then $\mathbb S^{m-1}_t (r^2)$ and\linebreak   $\mathbb H^{m-1}_{t-1}(-r^2)$, $m\geq 3$, are complete 
pseudo-Riemannian manifolds of constant curvature $r^2$ and $-r^{2}$, respectively.
The Lorentzian manifolds $\mathbb E^m_1$  and $\mathbb S^{m-1}_1 (r^2)$ 
are known as the Minkowski and de Sitter spaces, respectively.
For  $t=1$
$$
\mathbb H^{m-1}(-r^2)  =   \{x=(x_1, \dots,  x_m)
\in\mathbb E^m_1: \langle x, x\rangle=-r^{-2}\mbox{\ and \ } x_1>0\}
$$
is the hyperbolic space in $\mathbb E^m_1$.

The light cone  $\mathbf{\mathcal{LC}}^{n-1}$ with vertex at the origin in $\mathbb E^m_t$ is defined to be   
$$
\mathcal{LC}^{n-1}  =  \{x\in\mathbb E^m_t: \langle x,x\rangle=0\}.
$$

A vector $v$ in $\mathbb E^m_t$  is called space-like 
(resp., time-like) if $\langle v, v \rangle>0$  
(resp., $\langle  v, v \rangle<0$).  A vector $v$ is called light-like if it is nonzero and  it satisfies 
  $\langle v, v \rangle=0$.

Let $M$ be an $n$-dimensional pseudo-Riemannian submanifold  of the 
pseu\-do-Euclidean space $\mathbb E^m_t$. 
We denote Levi-Civita connections of $\mathbb E^m_t$ and $M$ by $\widetilde{\nabla}$ and $\nabla$,  
respectively. In this section, we shall use letters 
$X,\; Y,\;Z,\; W$ (resp., $\xi,\; \eta$) to denote vectors fields 
tangent (resp., normal) to $M$. The Gauss and Weingarten formulas are given, respectively, by
\begin{eqnarray}
\label{MEtomGauss} \widetilde\nabla_X Y&=& \nabla_X Y + h(X,Y),\\
\label{MEtomWeingarten} \widetilde\nabla_X \xi&=& -A_\xi(X)+D_X \xi,
\end{eqnarray}
 where $h$,  $D$  and  $A$ are the second fundamental form, the normal 
 connection and  the shape operator of $M$, respectively.
 
 For  each $\xi \in T^{\bot}_p M$, the shape operator $A_{\xi}$ is a symmetric 
 endomorphism of the tangent space  $T_p M$ at $p \in M$. 
 The shape operator and the second fundamental form are related by 
 $ \left\langle h(X, Y), \xi \right\rangle = \left\langle A_{\xi}X, Y \right\rangle$.
 
The Gauss, Codazzi and Ricci equations are given, respectively, by
\begin{eqnarray}
\label{MinkGaussEquation} \langle R(X,Y,)Z,W\rangle&=&\langle h(Y,Z),h(X,W)\rangle-
\langle h(X,Z),h(Y,W)\rangle,\\
\label{MinkCodazzi} (\bar \nabla_X h )(Y,Z)&=&(\bar \nabla_Y h )(X,Z),\\
\label{MinkRicciEquation} \langle R^D(X,Y)\xi,\eta\rangle&=&\langle[A_\xi,A_\eta]X,Y\rangle,
\end{eqnarray}
where  $R,\; R^D$ are the curvature tensors associated with connections $\nabla$ 
and $D$, respectively, and  $\bar \nabla h$ is defined by
$$(\bar \nabla_X h)(Y,Z)=D_X h(Y,Z)-h(\nabla_X Y,Z)-h(Y,\nabla_X Z).$$ 
A submanifold  $M$ is said to have flat normal bundle if $R^D=0$  identically, 
and the second fundamental form $h$ of $M$  in $\mathbb E^m_t$ is called  parallel if $\bar \nabla h=0$.
A submanifold with parallel second fundamental form is also known as  a parallel submanifold.

Let $\{e_1, e_2, \hdots, e_m\}$ be a local orthonormal frame on $M$ with 
$\varepsilon_A= \langle e_A, e_A\rangle$ $=\pm 1$ such that 
$e_1, e_2, \hdots, e_n$ are tangent to $M$ and $e_{n+1}, e_{n+2}, \hdots, e_m$ are normal to $M$. 
We use  the following convention on the range of indices:  
$
1\leq A,B,C,\ldots \leq m,\quad 1\leq i,j,k,\ldots \leq n, \quad 
n+1\leq \beta,\gamma, \ldots \leq m.
$

Let  $\{\omega_{AB}\}$ with $\omega_{AB} + \omega_{BA} =0$ be the
connection 1-forms associated to $\{e_1,\dots,e_{m} \}$. Then we have 
$$
\widetilde{\nabla}_{e_k}e_i= \sum_{j=1}^{n} \varepsilon_j\omega_{ij}(e_k) e_j + 
\sum_{\beta =n+1}^{m} \varepsilon_{\beta}  h^{\beta}_{ik}   e_{\beta} 
$$
and
$$
 \widetilde{\nabla}_{e_k}e_{\beta}= - \sum_{j=1}^{n} \varepsilon_j h^{\beta}_{kj} e_j+
  \sum_{\nu =n+1}^{m} \varepsilon_{\nu} \omega_{\beta\nu } (e_k) e_{\nu},
$$  
 where $h^{\beta}_{ij}$'s are the coefficients of the second fundamental form $h$.

The mean curvature vector $H$, the scalar curvature $S$ and  the squared length $\| h\|^2$ of the second 
fundamental form $h$ are defined by
\begin{equation} \label{MinkOrtEgTanim}
H=\frac 1n \sum\limits_{\beta=n+1}^m \varepsilon_\beta\mathrm{tr}A_\beta e_\beta,
\end{equation}
\begin{equation}
\label{Mink2esFormUzTanim} \| h \|^2 =
\sum\limits_{i,j, \beta}\varepsilon_i\varepsilon_j\varepsilon_\beta h^\beta_{ij} h^\beta_{ji},
\end{equation} 
\begin{equation} \label{Scaler-curvature}
S = n^2 \langle H, H \rangle - \| h \|^2,
\end{equation}
where $\mathrm{tr}A_\beta$ denotes the trace of shape operator $A_\beta$, i.e.,
$
\mathrm{tr}A_\beta =  \sum\limits_{i=1}^n \varepsilon_i h^\beta_{ii}.
$

 The mean curvature vector $H$ of a submanifold of $M$ in $\mathbb E^m_t$ is called parallel if $DH=0$ identically.

The gradient of a smooth function $f$ defined on $M$ into $\mathbb R$ is defined by
$
\nabla f  =  \sum\limits_{i=1}^n  \varepsilon_i e_i(f)e_i
$
and the Laplace operator acting on $M$ is $
\Delta =\sum\limits_{i=1}^n\varepsilon_i(\nabla_{e_i}e_i-e_ie_i).
$
If the position vector $x$ of $M$ in $E^m_s$ satisfies 
$\Delta x\neq0$ and $\Delta^2 x=0$, then $M$ is called biharmonic.

 A surface $M$ in $\mathbb E^4_1$ is called space-like if 
 every non-zero tangent vector on $M$ is space-like. Let $\{e_1,e_2,e_3,e_4\}$ 
 be a local  orthonormal frame  on a space-like surface $M$ such that $e_1, e_2$ are tangent to $M$ 
 and $e_3, e_4$ are normal to $M$ with $\varepsilon_\beta=\langle e_\beta,e_\beta\rangle,\ \beta=3,4$.

The Gaussian curvature $K$  is defined by $K=R(e_1,e_2;e_2,e_1)$. 
Note that scalar curvature $S$ and Gaussian curvature of $M$  satisfies $S=2K$. 
Thus, \eqref{Scaler-curvature} implies  
\begin{equation} \label{Gauss-curvature} 
 K =   2 \langle H, H \rangle - \| h \|^2 /2.
\end{equation}
From Gauss equation \eqref{MinkGaussEquation} we have  $K=\varepsilon_3(\det A_3-\det A_4)$. 
If $K$  vanishes identically, $M$ is said to be flat.  
On the other hand, $M$ is called maximal if $H=0$. A surface $M$ is called pseudo-umbilical if its second fundamental form $h$ and the mean curvature vector $H$ satisfies
$\langle h(X,Y),H\rangle=\rho\langle X,Y\rangle$
for a smooth function $\rho$. Moreover, if the equation 
$h(X,Y)=\langle X,Y\rangle H$
is satisfied, then $M$ is said to be totally umbilical.

If we put $h_{ij,k} = (\nabla_{e_k} h )(e_i, e_j)$, then for a space-like surface $M$ in $\mathbb E^4_1$ 
the Codazzi equation given by \eqref{MinkCodazzi} becomes
\begin{align} \label{MinkCodazziversion2}
\begin{split}
&h^\beta_{ij,k}= h^\beta_{jk,i},\quad i,j,k=1,2,\ \beta=3,4\\
&h^\beta_{jk,i} = e_i(h^\beta_{jk}) + 
\sum_{\gamma=3}^4\varepsilon_\gamma h^\gamma_{jk} \omega_{\gamma\beta} (e_i) -
\sum_{\ell=1}^2 \left ( \omega_{j\ell}(e_i) h^\beta_{\ell k} + 
\omega_{k\ell}(e_i) h^\beta_{j \ell} \right ).
\end{split}
\end{align}

Let $G(m-n, m)$ be the Grassmannian manifold consisting of 
all oriented $(m-n)$-planes through the origin of $\mathbb E^m_t$ 
and $\bigwedge^{m-n} \mathbb E^m_t$  the vector space obtained by 
the exterior product of  $m-n$ vectors in $\mathbb E^m_t$.
Let  $f_{i_1} \wedge \cdots \wedge f_{i_{m-n}}$ and $g_{i_1} 
\wedge \cdots \wedge g_{i_{m-n}}$ be two vectors in $\bigwedge^{m-n} \mathbb E^m_t$,
 where  $\{f_1, f_2, \dots, f_m\}$ and $\{g_1, g_2, \dots, g_m\}$ 
 are two orthonormal bases of $\mathbb E^m_t$.
 Define an indefinite inner product $\left\langle,  \right\rangle$ on 
 $\bigwedge^{m-n} \mathbb E^m_t$ by 
\begin{equation} \label{inner-prod}
\left\langle  f_{i_1} \wedge \cdots \wedge f_{i_{m-n}},  g_{i_1} \wedge \cdots \wedge g_{i_{m-n}} \right\rangle
= \det( \left\langle f_{i_\ell}, g_{j_k}\right\rangle).
\end{equation}
Therefore, for some positive integer $s$, we may identify $\bigwedge^{m-n} \mathbb E^m_t$ with 
some pseudo-Euclidean space $\mathbb E^N_s$,
where $N= {m\choose m-n}$.  
Let $ e_1,\dots,e_{n}, e_{n+1}, \dots,e_m$ be an oriented local orthonormal frame 
on an $n$-dimensional  pseudo-Riemann\-ian submanifold  $M$
in  $\mathbb E^m_t$ with $\varepsilon_B=  \langle e_B, e_B\rangle =\pm 1$ such that 
  $ e_1,\dots,e_{n}$ are tangent to $M$ and $  e_{n+1}, \dots, e_m$  are normal to $M$.
The map $\nu : M \rightarrow G(m-n, m) \subset   \mathbb E^N_s$  
from an oriented pseudo-Riemannian submanifold $M$ into 
$G(m-n, m)$ defined by 
\begin{equation}\label{MinkGaussTasvTanim}
\begin{array}{rcl}
\nu(p) = (e_{n+1} \wedge e_{n+2} \wedge \cdots \wedge e_{m}) (p)
\end{array}
\end{equation}
 is called the {\it Gauss map} of $M$ that is a smooth map 
 which assigns to a point $p$ in $M$  the oriented 
$(m-n)$-plane through  the origin of $\mathbb E^m_t$ and parallel  
to the normal  space of $M$ at $p$, \cite{Kim-Yoon-2}.
We put $\varepsilon =  \left\langle \nu, \nu \right\rangle = 
\varepsilon_{n+1} \varepsilon_{n+2} \cdots \varepsilon_m = \pm 1$ and
 \begin{equation}\label{l-like-n2-cmc}
\widetilde M^{N-1}_s (\varepsilon) =\left\{ \displaystyle 
 \begin{array}{lll} 
\displaystyle \mathbb S^{N-1}_s (1) \;\;& \mbox{in} \; \;\mathbb E^N_s & \mbox{if} \; \; \varepsilon =1
  \\ \notag
\displaystyle \mathbb H^{N-1}_{s-1} (-1) \;\;& \mbox{in} \;
 \; \mathbb E^N_s & \mbox{if} \; \; \varepsilon=-1.
\end{array}
\right.
    \end{equation}
Then the Gauss image $\nu(M)$ can be viewed as 
$\nu(M) \subset \widetilde M^{N-1}_s (\varepsilon)$.

\section{Space-like  surfaces in $\mathbb E^4_1$ with harmonic Gauss Map} \label{SectionHarmonic}

The Laplacian of the Gauss map of an $n$-dimensional oriented 
submanifold $M$ of a Euclidean space $\mathbb E^{n+2}$ was obtained in \cite{Dursun-Arsan}.
 By a similar calculation, for the Laplacian of the Gauss map $\nu$ given by 
 \eqref{MinkGaussTasvTanim} of an $n$-dimensional 
 oriented submanifold $M$ of a pseudo-Euclidean space $\mathbb E^{n+2}_t$ we have 

\begin{Lemma}\label{MinkGaussLapLem}
Let $M$ be an $n$-dimensional oriented submanifold of a pseudo-Euclidean space $\mathbb E^{n+2}_t$. Then 
 the Laplacian of Gauss map  $\nu = e_{n+1}\wedge e_{n+2} $ is given by 
\begin{eqnarray}
\nonumber \Delta\nu&=&\| h\|^2\nu+2\sum\limits_{1\leq j<k\leq n}\varepsilon_j\varepsilon_k
R^D (e_j,e_k;e_{n+1},e_{n+2}) e_j\wedge e_k\\&& +\nabla(\mathrm{tr}A_{n+1})\wedge e_{n+2}  \label{MinkGaussLaplEnSon}
+e_{n+1}\wedge 
 \nabla(\mathrm{tr}A_{n+2})\\&&\nonumber+n\sum\limits_{j=1}^n\varepsilon_j\omega_{(n+1)(n+2)} (e_j) H\wedge e_j,
\end{eqnarray}
where ${\| h \|}^2$ is the squared length of the second
fundamental form, $R^{D} \,$ the normal curvature tensor and 
$\nabla\tr A_r$  the gradient of $\tr A_r$.
\end{Lemma}

\begin{Remark}\label{MinRemark1}
From \eqref{MinkGaussLaplEnSon} we see that if an $n$-dimensional submanifold $M$ of 
$\mathbb  E^{n+2}_t$ has pointwise 1-type Gauss map of the first kind, then  equation 
\eqref{PW1TypeDefinition} is satisfied for $f=\| h\|^2$ and $C=0$.
\end{Remark}

The Gauss map of a surface $M$ in $\mathbb E^4_1$  is said to be harmonic 
if $\Delta \nu = 0$. 
Clearly, a harmonic Gauss map is of (global) 1-type of the first kind. In the Euclidean space 
$\mathbb E^4$, a plane is the only surface with harmonic Gauss map. However, in the Minkowski space 
$\mathbb E^4_1$ there are non-planar surfaces   with harmonic Gauss map.

\begin{Lemma} \label{MinkYuzOrtEgrVektPariseNormKnDuzver1} \cite{ChenCentEur2009} 
Let $M$ be a space-like surface with parallel mean curvature vector $H$ in $\mathbb E^4_1$. Then we have:
\begin{enumerate}
\item[(a)] $\langle H,H\rangle$ is constant and 
\item[(b)]  $[A_H,A_\xi]=0$ for any normal vector field $\xi$.
\end{enumerate}
\end{Lemma}
By combining the part (b) of Lemma \ref{MinkYuzOrtEgrVektPariseNormKnDuzver1} and the 
Ricci equation \eqref{MinkRicciEquation} we state the following lemma for later use:

\begin{Lemma}\label{MinkYuzOrtEgrVektPariseNormKnDuz}
Let $M$ be a non-maximal space-like  surface in the Minkowski space $\mathbb E^4_1$. If the mean
curvature vector $H$ of $M$ is  parallel, then the normal bundle of $M$ is flat, i.e., $R^D \equiv0$.
\end{Lemma}
 

\begin{Prop} \label{PropMax3kosuldenk}
Let $M$ be an oriented maximal surface in the Minkowski space  $\mathbb E^4_1$. 
Then the  Gauss map $\nu$ of $M$ is harmonic if and only if 
$M$ is a flat surface in    $\mathbb E^4_1$ with flat normal bundle.
\end{Prop}
\begin{proof}
Let $M$ be a  maximal surface in $\mathbb E^4_1$, i.e., $H\equiv 0$. Then, from \eqref{Gauss-curvature} we have $\|h\|^2=-2K$.
Thus, \eqref{MinkGaussLaplEnSon} implies
\begin{eqnarray}
\label{MaksYuzeyE41GausLapl}\Delta\nu&=&-2K\nu+2R^D (e_1,e_2;e_3,e_4) e_1\wedge e_2.
\end{eqnarray}
Therefore, $\nu$ is harmonic if and only if $K=0$ and $R^D=0$. Hence the proof is completed.
\end{proof}

Next, we obtain a non-planar maximal surface in  $\mathbb E^4_1$ with harmonic Gauss map. 
\begin{Example} \cite{Chen-Ishikawa}
Let $\Omega$ be an open, connected set in $\mathbb R^2$ and 
$\phi:\Omega\rightarrow\mathbb R$  a smooth function. 
We consider the surface $M$ in the Minkowski space $\mathbb E^4_1$ given by
\begin{eqnarray}\label{MinkMargTrapnottooComplete}
x(u,v) = (\phi (u, v), u, v,\phi (u, v)).
\end{eqnarray}
This surface lies in  the degenerate hyperplane 
$\mathcal{H}_0=\{(x_1,x_2,x_3,x_4)\in \mathbb E^4_1\ |\ x_1=x_4 \}$.
By a direct calculation, we see that $M$ is a flat surface with 
flat normal bundle and the mean curvature 
vector $H$ of $M$ in  $\mathbb E^4_1$ is given by  
\begin{equation}\label{MinkMargTrapnottooCompleteH}
H=(\Delta  \phi,0,0,\Delta \phi).
\end{equation}
Therefore, $M$ is maximal if and only if $\phi$ is harmonic. 

Hence, Proposition \ref{PropMax3kosuldenk} implies that if  
$\phi$ is a harmonic function, then the surface given by 
\eqref{MinkMargTrapnottooComplete} has harmonic Gauss map.
\end{Example}

\begin{Prop}\label{MinkMaksNormalDemDuzClassTheorem}
A non-planar flat   maximal surface  in the Minkowski space  $ \mathbb E^4_1$ with flat normal bundle is  
congruent to the surface given by \eqref{MinkMargTrapnottooComplete} for a smooth  harmonic function 
$\phi:\Omega\subset\mathbb R^2\rightarrow\mathbb R$, where $\Omega$ is an open set in $\mathbb R^2$.
\end{Prop}
\begin{proof}
Let $M$ be a  non-planar flat maximal surface  in $\mathbb  E^4_1$ with flat normal bundle. 

Since $M$ is flat, there exist local coordinates $(u, v)$ on $M$ such that $e_1=\partial/\partial u$, $e_2=\partial/ \partial v$.
Then the induced metric tensor is given by  $g=du^2+dv^2$ and also $\omega_{12}\equiv 0$. Let $\{e_3, e_4\}$ be 
a local orthonormal normal frame  on $M$ with $\varepsilon_3=-\varepsilon_4=1$, where $\varepsilon_{\beta} = 
\left\langle e_{\beta}, e_{\beta} \right\rangle $. As $M$ is maximal,  we have 
$A_\beta=(h^\beta_{ij})$ with $h^\beta_{11} + h^\beta_{22} =0 $, $\beta=3,4$, that is,  
$\mathrm{tr}A_3=\mathrm{tr}A_4 =0$. Moreover, since $M$ is flat, from  the Gauss equation \eqref{MinkGaussEquation}
we get  $\det A_3=\det A_4$. 
Therefore,  the eigenvalues of $A_3$ and $A_4$ are equal which imply that $A_3=\mp A_4$ as  $R^D=0$. 
Without loss of generality, we may take  $A_3=A_4$. 

Let $\Omega$ be an open set in $\mathbb R^2$ and  $x:\Omega\rightarrow M\subset \mathbb  E^4_1$ 
be an isometric immersion. From the Gauss formula we obtain that
\begin{equation} \label{MinkMaksNormalDemDuzClassGauss1}  
x_{uu}= h^3_{11}(e_3-e_4), \quad  x_{uv}= h^3_{12}(e_3-e_4), \quad 
 x_{vv}= -h^3_{11}(e_3-e_4)
\end{equation}
 as $\omega_{12}\equiv 0$. Also, the first and second equations in \eqref{MinkMaksNormalDemDuzClassGauss1} 
   imply that
\begin{equation}\label{MinkMaksNormalDemDuzClassDenk1}
x_{uu}+x_{vv}=0.
\end{equation}
Moreover, $x_{uu}$, $x_{uv}$ and $x_{vv}$ are pairwise linearly dependent light-like vector fields. 

On the other hand, by a direct calculation, using the Weingarten formula and 
\eqref{MinkMaksNormalDemDuzClassGauss1}, we obtain  
\begin{eqnarray}
\label{MinkMaksNormalDemDuzClassEq1} x_{uuu}&=&\left(\partial_u\left(h^3_{11}\right)
+\omega_{34}(\partial_u)h^3_{11}\right)(e_3-e_4),\\
\label{MinkMaksNormalDemDuzClassEq2} x_{uuv}&=&\left(\partial_v\left(h^3_{11}\right)
+\omega_{34}(\partial_v)h^3_{11}\right)(e_3-e_4).
\end{eqnarray}

Now we define a vector valued function $y=(y^1,y^2,y^3,y^4):\Omega\rightarrow \mathbb E^4_1$ as $y=x_{uu}$. 
From equations  \eqref{MinkMaksNormalDemDuzClassGauss1},  \eqref{MinkMaksNormalDemDuzClassEq1} 
and \eqref{MinkMaksNormalDemDuzClassEq2} 
we obtain  $y_u=\gamma_1 y$ and $y_v=\gamma_2 y$ for some smooth functions $\gamma_1$
 and $\gamma_2$. Thus, the coordinate functions of $y$ satisfy
\begin{eqnarray}
 y^i_u=\gamma_1 y^i, \quad y^i_v=\gamma_2 y^i,\quad i=1,2,3,4
\end{eqnarray}
By solving these equations, we get $y^j=c_jy^1,\;i=2,3,4$, for some constants $c_j\in\mathbb R$. 
Thus, we have
\begin{eqnarray}
\label{MinkMaksNormalDemDuzClassDenk2} x_{uu}&=&y^1 \mathbf{\eta}_0,
\end{eqnarray}
where $\mathbb{\eta}_0=(1,c_2,c_3,c_4)$ is a constant light-like  vector. 
In a similar way, we obtain 
\begin{equation} \label{MinkMaksNormalDemDuzClassDenk3} 
x_{uv}= \phi_2 \mathbf{\eta}_0  \quad \mbox{and}    \quad  x_{vv}  =   \phi_3 \mathbf{\eta}_0, 
\end{equation} 
where $\phi_2,\phi_3:\Omega \subset \mathbb R^2  \rightarrow\mathbb R$ are some smooth functions. 
By integrating \eqref{MinkMaksNormalDemDuzClassDenk2} and \eqref{MinkMaksNormalDemDuzClassDenk3}, 
 we obtain  
 $$  
 x(u, v) =\phi (u, v) \mathbf{\eta}_0+u\mathbf{\eta}_1+ v \mathbf{\eta}_2,
 $$ 
where $\phi:\Omega\rightarrow\mathbb R$ is  a smooth function and 
$\mathbf{\eta}_1$, $\mathbf{\eta}_2$ are constant vectors  
such that $ \left\langle \eta_0, \eta_i \right\rangle = 0$, 
$ \left\langle \eta_i, \eta_j \right\rangle = \delta_{ij}, \; i, j=1, 2.$
Equation \eqref{MinkMaksNormalDemDuzClassDenk1} implies  that $\phi$ is harmonic. By choosing 
 $\mathbf{\eta}_0=(1,0,0,1)$, $\mathbf{\eta}_1=(0,1,0,0)$ and 
 $\mathbf{\eta}_2=(0,0,1,0)$, the proof is completed. 
\end{proof}
By combining Proposition \ref {PropMax3kosuldenk} and 
Proposition \ref{MinkMaksNormalDemDuzClassTheorem} we state the following classification theorem 
for maximal surfaces in $ \mathbb E^4_1$ with harmonic Gauss map:
\begin{theorem} \label{PropMax3kosulTheorem}
An oriented  maximal surface with harmonic Gauss map in the Minkowski space $ \mathbb E^4_1$ is either 
an open part of a space-like plane or congruent to  a surface
given by \eqref{MinkMargTrapnottooComplete} for a smooth harmonic function 
$\phi:\Omega\subset\mathbb R^2\rightarrow\mathbb R$, where $\Omega$ is an open set in $\mathbb R^2$.
\end{theorem}
Now we investigate non-maximal space-like surfaces in $ \mathbb E^4_1$ with harmonic Gauss map.
\begin{theorem}\label{MinkNonMaksHarmonikTheo}
Let $M$ be an oriented  non-maximal space-like surface in the Minkowski space $\mathbb E^4_1$. 
Then the Gauss map $\nu$ of $M$ is harmonic if and only if $M$ is flat in $\mathbb E^4_1$ 
with  light-like and parallel  mean curvature vector.
\end{theorem}
\begin{proof}
Let $M$ be an oriented  non-maximal space-like surface in  $\mathbb E^4_1$ with harmonic Gauss map $\nu$. 
Then we have $\Delta\nu=0$. From  \eqref{MinkGaussLaplEnSon} we obtain $\|h\|^2=0$ and $R^D=0$. 
That is, the normal bundle is flat. So we can choose a local parallel orthonormal 
normal frame $\{e_3, e_4\}$ on $M$. Thus  we have $\omega_{34}=0$, and from \eqref{MinkGaussLaplEnSon} 
\begin{eqnarray}
\label{Minkhkareesit0nonmakCod0}
\nabla(\mathrm{tr}A_3)\wedge e_{4}+e_3\wedge \nabla(\mathrm{tr}A_4)=0
\end{eqnarray}
which implies that  $\mathrm{tr}A_3$ and $\mathrm{tr}A_4$ are constants. Therefore, $DH=0$, that is, 
$H$ is parallel, and $\langle H,H\rangle $ is constant because of part (a) of Lemma \ref{MinkYuzOrtEgrVektPariseNormKnDuzver1}.

Now we will show that $H$ is light-like. Suppose that $H$ is not  light-like, that is, $\langle H,H\rangle\neq 0$.
As $\|h\|^2=0$ we have  $K=2\langle H,H\rangle \not = 0$  from \eqref{Gauss-curvature}. Thus,  $M$ is not flat.

On the other hand, since the normal bundle is flat and $\langle H,H\rangle\neq 0$,  
we can choose a local orthonormal frame field $\{e_1,e_2,e_3,e_4\}$ on $M$ such that 
$e_3= H/\alpha$ and $e_4$ are parallel,   the shape operators are diagonalized, and $e_1, e_2$ are eigenvectors of $A_3$, 
where $\alpha=\sqrt{|\langle H, H |\rangle}$.
So we have 
$A_3=\mathrm{diag} \left(h^3_{11}, h^3_{22}\right),\quad A_4=\mathrm{diag} \left(h^4_{11}, -h^4_{11}\right),$
$ h^3_{11}+h^3_{22}=2\alpha $   and  $\omega_{34}=0$. Considering these, it follows from 
Codazzi equation \eqref{MinkCodazziversion2} that 
\begin{eqnarray}
\label{Minkhkareesit0nonmakCod1} e_1(h^3_{22})=-e_1(h^3_{11})&=&\omega_{12}(e_2)\Big(h^3_{11}-h^3_{22}\Big),\\
\label{Minkhkareesit0nonmakCod2} e_1(h^4_{22})=-e_1(h^4_{11})&=&2\omega_{12}(e_2)h^4_{11},\\
\label{Minkhkareesit0nonmakCod3} e_2(h^3_{11})=-e_2(h^3_{22})&=&\omega_{12}(e_1)\Big(h^3_{11}-h^3_{22}\Big),\\
\label{Minkhkareesit0nonmakCod4} e_2(h^4_{11})=-e_2(h^4_{22})&=&2\omega_{12}(e_1)h^4_{11}.
\end{eqnarray}

As $\|h\|^2=0$  we have
\begin{eqnarray}
\label{Minkhkareesit0nonmakeYildiz} \big(h^3_{11}\big)^2+\big(h^3_{22}\big)^2=2\big(h^4_{11}\big)^2
\end{eqnarray}
from which we obtain
\begin{eqnarray} \notag 
 h^3_{11}e_1(h^3_{11})+h^3_{22}e_1(h^3_{22})&=&2h^4_{11}e_1(h^4_{11}),\\ \notag
 h^3_{11}e_2(h^3_{11})+h^3_{22}e_2(h^3_{22})&=&2h^4_{11}e_2(h^4_{11}).
\end{eqnarray}
Using  \eqref{Minkhkareesit0nonmakCod1}-\eqref{Minkhkareesit0nonmakCod4}, the above equations  become
\begin{eqnarray}
\label{Minkhkareesit0nonmake1_1} -\omega_{12}(e_2)\Big((h^3_{11}-h^3_{22})^2-4(h^4_{11})^2\Big)=&0,\\
\label{Minkhkareesit0nonmake2_1} \omega_{12}(e_1)\Big((h^3_{11}-h^3_{22})^2-4(h^4_{11})^2\Big)=&0.
\end{eqnarray}
Since $M$ is not flat, at least one of  $\omega_{12}(e_1)$ and  $\omega_{12}(e_2)$ is not zero.
Therefore,  \eqref{Minkhkareesit0nonmake1_1} and  \eqref{Minkhkareesit0nonmake2_1} imply that 
 $$(h^3_{11}-h^3_{22})^2=4(h^4_{11})^2.$$ 
Considering this and  \eqref{Minkhkareesit0nonmakeYildiz}  we obtain that  $h^3_{11}h^3_{22}+(h^4_{11})^2=0$.
Therefore, the Gauss curvature $K=\varepsilon_3(h^3_{11}h^3_{22}+(h^4_{11})^2)=0$ and hence $2\langle H,H\rangle  =K = 0$ 
which is a contradiction. As a result, $H$ is light-like. Since $\langle H,H\rangle=0$ and $\|h\|^2=0$, 
\eqref{Gauss-curvature} implies $K=0$, i.e. $M$ is flat.

Conversely, we assume that $M$ is a flat surface in  $\mathbb E^4_1$ with parallel and 
light-like mean curvature vector $H$, that is,  $K=\langle H,H\rangle=0$.
So  we have $\|h\|^2=0$ from \eqref{Gauss-curvature}. 
On the other hand, Lemma \ref{MinkYuzOrtEgrVektPariseNormKnDuz} 
implies that $M$ has flat normal bundle, i.e., $R^D=0$. 
Therefore, there exists a local parallel orthonormal frame $\{e_3,e_4\}$ of normal bundle of 
$M$ with $ \varepsilon_3=-\varepsilon_4=1$ and the shape operators 
$A_3$ and $A_4$ can be diagonalized simultaneously  by choosing 
a proper frame $\{e_1,e_2\}$ of tangent bundle of $M$, namely, we have  
$$A_\beta=\mathrm{diag} \left(h^\beta_{11},h^\beta_{22}\right),\quad \beta=3,4,$$
also, $\omega_{34}\equiv0$. 
Moreover, since $H$ is light-like, we have
$$\mathrm{tr}A_3=\mathrm{tr}A_4=\mu\neq 0\quad\mbox{and}\quad H=\frac \mu2(e_3-e_4).$$ 
In addition, since $H$ is parallel and $\omega_{34}=0$, $\mu$ is a constant.  
Thus, we have $\nabla(\mathrm{tr}A_3)=\nabla(\mathrm{tr}A_4)=0.$
Therefore, equation  \eqref{MinkGaussLaplEnSon} gives $\Delta\nu=0$. 
\end{proof}
A space-like surface in the Minkowski space $\mathbb E^4_1$ is called {\it marginally trapped} (or quasi-minimal) if its mean curvature 
vector is light-like at each point on the surface. We will use the following classification theorem of  marginally trapped surfaces with parallel mean curvature vector in 
the Minkowski space $\mathbb E^4_1$ obtained in \cite{ChenVeken2009Houston}.
\begin{theorem}   \label{TheoremChenVeken2009Houston} 
\cite{ChenVeken2009Houston} Let $M$ be a marginally trapped surface with parallel mean curvature vector in 
the Minkowski space-time $\mathbb E^4_1$. Then, with respect to suitable 
Minkowskian coordinates $(t,x_2,x_3,x_4)$ on $\mathbb E^4_1$, $M$ is an open part of 
one of the following six types of surfaces:
\begin{enumerate}
\item[(i)] a flat parallel biharmonic surface given by
$$x(u,v)=\left(\frac{1-b}{2}u^2+\frac{1+b}{2}v^2, u, v, \frac{1-b}{2}u^2+
\frac{1+b}{2}v^2\right),b\in\mathbb R;$$
\item[(ii)] a flat parallel surface given by
\begin{equation}
\label{TheoremChenVeken2009_Case1}
x(u, v)=a(\cosh u,\sinh u,\cos u,\sin u),\; a>0;
\end{equation}
\item[(iii)] a non-parallel flat biharmonic surface with constant light-like 
mean curvature vector, lying in the hyperplane $\mathcal{H}_0=\{(t,x_2,x_3,t)\},$ 
but not in the light cone $\mathcal{LC}$;
\item[(iv)] a non-parallel flat  surface  lying in the light cone 
$\mathcal{LC}$; 
\item[(v)] a non-parallel  surface  lying in the de Sitter space-time $S^3_1(r^2)$ 
for some $r>0$ such that the mean curvature vector $H'$ of $M$ in $S^3_1(r^2)$  satisfies 
$\langle H',H'\rangle=-r^2$;
\item[(vi)] a non-parallel surface  lying in the hyperbolic space $H^3(-r^2)$ for some 
$r>0$ such that the mean curvature vector $H'$ of $M$ in $H^3(-r^2)$  satisfies $\langle H',H'\rangle=r^2$.
\end{enumerate}
Conversely, all surfaces of type (i)-(vi) above give rise to marginally trapped 
surfaces with parallel mean curvature vector in $\mathbb E^4_1$.
\end{theorem}
\begin{Remark}\cite{ChenJmathPhys2009}\label{RemarkChenJmathPhys2009}
We can combine cases (i) and (iii) of Theorem \ref{TheoremChenVeken2009Houston} 
into a single case, namely, flat surfaces defined by \eqref{MinkMargTrapnottooComplete} 
such that $\phi$ is a function satisfying $\Delta \phi=c$ for some real
number $c\neq 0$.
\end{Remark}
The surfaces type (i) and (ii) in Theorem \ref{TheoremChenVeken2009Houston}  
are two explicit  examples for Theorem \ref{MinkNonMaksHarmonikTheo}.
In the next theorem  we determine flat surfaces in $\mathbb S^3_1(r^2)\subset\mathbb E^4_1$ 
with parallel and light-like mean curvature vector in $\mathbb E^4_1$.
\begin{theorem}\label{LemmaAboutTheoremChenVeken2009Houston-1}
Let $M$ be a space-like surface in the de Sitter space $\mathbb S^3_1(r^2)\subset\mathbb E^4_1$ for some $r>0$. 
If $M$ is a flat surface with parallel and light-like mean curvature vector 
in $\mathbb E^4_1$, then $M$ is congruent to the surface given by
\begin{eqnarray}\label{MinkFlatMargTrapS31}
x(u,v) =  \left(\frac{r}{2}(u^2+v^2), u, v, \frac{r}{2}(u^2+v^2)-\frac 1r \right).
\end{eqnarray}
\end{theorem}
\begin{proof}
Suppose that   $M$ is a flat  space-like surface in  
$ \mathbb S^3_1(r^2)\subset\mathbb E^4_1$ with parallel and light-like mean curvature vector $H$
in $\mathbb E^4_1$. 
Since $M$ is flat,  there exist  local coordinates  $u$ and $v$ on $M$ 
such that the induced metric tensor is $g=du^2+dv^2$. 
Let $x:\Omega\rightarrow M\subset\mathbb S^3_1(r^2)\subset\mathbb E^4_1$ 
be an isometric immersion, where $\Omega$ is an open set in $\mathbb R^2$.  
Then, we have $\langle x,x\rangle=r^{-2}$. Thus, a local frame field 
$\{e_1,e_2, e_3, e_4\}$ on $M$ can be chosen as  $e_1=\partial_u$, $e_2=\partial_v$, 
$e_3= r x $, and  $e_4$ is a unit normal vector field orthogonal to $e_3$ such that 
$H=-r(e_3-e_4)$  as $H$ is light-like (see \cite[Lemma 2.2]{ChenCentEur2009}).

From  the Weingarten formula \eqref{MEtomWeingarten}, we have  
$\widetilde\nabla_{\partial_u}e_3=r\partial_u$ and $\widetilde\nabla_{\partial_v}e_3=r\partial_v$ 
which imply $A_3=-rI$, where $I$ is identity operator acting on tangent bundle of $M$. 
Moreover, since $M$ is flat and $H=-r(e_3-e_4)$, we have 
$\det A_3=\det A_4$, $\mathrm{tr} A_3=\mathrm{tr} A_4$ from which  and $A_3=-rI$ we obtain $A_3=A_4$.  Thus, $M$ is pseudo-umbilical. \cite[Theorem 4]{ChenArchMat1994} implies $M$ is biharmonic. 

As $M$ is a biharmonic surface with light-like mean curvature vector, from the proof of \cite[Theorem 6.1]{Chen-Ishikawa}, one can see that $x$ is of the form
\begin{equation}
\label{xphiphic} x(u,v)=(\phi(u,v),u,v,\phi(u,v)-\phi_0)
\end{equation}
for a smooth function $\phi$ and constant $\phi_0\neq0$. As $\langle x,x\rangle=r^{-2}$, from \eqref{xphiphic} we obtain
$$-2\phi_0\phi+\phi_0^2+u^2+v^2=r^{-2}.$$
By considering this equation and a linear isometry of $\mathbb E^4_1$, we may assume that
\begin{equation}
\label{xphiphic2}\phi(u,v)=\frac{r}{2}(u^2+v^2)\quad\mbox{and}\quad \phi_0=\frac 1r
\end{equation}
from which and \eqref{xphiphic} we obtain \eqref{MinkFlatMargTrapS31}.
\end{proof}
Similarly we state 
\begin{theorem}\label{LemmaAboutTheoremChenVeken2009Houston-2}
Let $M$ be a space-like surface in the hyperbolic space $\mathbb H^3(-r^2)\subset\mathbb E^4_1$ for some $r>0$. 
If  $M$ is a flat surface with parallel and light-like mean curvature vector 
in $\mathbb E^4_1$, then   $M$ is congruent to the surface given by
\begin{eqnarray}\label{MinkFlatMargTrapH3}
x(u,v) = \left(\frac 1r + \frac{r}{2}(u^2+v^2), u, v, \frac{r}{2}(u^2+v^2) \right).
\end{eqnarray}
\end{theorem}
The proof of this theorem is similar to the proof of Theorem \ref{LemmaAboutTheoremChenVeken2009Houston-1}.
\begin{Corol}
 Up to linear isometries in $\mathbb E^4_1$, the surface given by \eqref{MinkFlatMargTrapS31} 
(resp., \eqref{MinkFlatMargTrapH3}) is the only surface  in $\mathbb S^3_1(r^2)\subset\mathbb E^4_1$ 
(resp., $\mathbb H^3(r^2)\subset\mathbb E^4_1$) with harmonic  Gauss map.
\end{Corol}
By combining the results given in this section, we state
\begin{theorem}\label{ClassificationTheorem1typenonHarmo}
Let $M$ be an oriented space-like surface in the Minkowski space $\mathbb {E}^4_1$. 
Then the Gauss map $\nu$ of $M$ is harmonic  if and only if $M$ is congruent 
to one of the following six types of surfaces:
\begin{enumerate}
\item[(i)] an open part of a space-like plane;
\item[(ii)] the flat surface given by \eqref{MinkMargTrapnottooComplete} 
for a smooth function $\phi:\Omega\rightarrow \mathbb R$ satisfying $\Delta\phi=c$, 
where $\Omega$ is an open set in $\mathbb R^2$ and $c\in\mathbb R$;
\item[(iii)] the flat surface given by \eqref{TheoremChenVeken2009_Case1};
\item[(iv)] a non-parallel flat surface lying in the light cone $\mathcal{LC}$;
\item[(v)] the flat surface given by \eqref{MinkFlatMargTrapS31} 
lying in the de-Sitter space-time $\mathbb S^3_1(r^2)$;
\item[(vi)] the flat surface given by \eqref{MinkFlatMargTrapH3} 
lying in the hyperbolic space $\mathbb H^3(-r^2)$.
\end{enumerate}
\end{theorem}

\section{Space-like  surfaces in $\mathbb E^4_1$ with pointwise 1-type 
Gauss map of the first kind} \label{SectionPW1Type}

Let $M$ be an oriented space-like surface in the Minkowski space $\mathbb E^4_1$ 
with harmonic Gauss map $\nu$. 
Then $\nu$ satisfies \eqref{PW1TypeDefinition} for $f=0$ and $C=0$. 
Thus, a harmonic Gauss map $\nu$ is of pointwise 1-type of the first kind. 
In this section, we obtain a  characterization of 
surfaces  in  $\mathbb E^4_1$ with pointwise 1-type Gauss map of the first kind.

\begin{theorem}\label{MinMinimal1Cesit1tip}
Let $M$ be an oriented  maximal surface in the Minkowski space $\mathbb E^4_1$. 
Then $M$ has pointwise 1-type Gauss map of the first kind if and only if 
$M$ has  flat normal bundle. Moreover, the Gauss map $\nu$ satisfies  
\eqref{PW1TypeDefinition} for  $f=\|h\|^2$ and $C=0$.
\end{theorem}

\begin{proof} 
If $M$ is maximal, then the  Gauss map $\nu$   satisfies \eqref{MaksYuzeyE41GausLapl}. 
Hence, $\nu$ is of  pointwise 1-type of the first kind if and only if $R^D=0$. 
\end{proof}

We now give the following lemma.
\begin{Lemma}\label{Lemma1MinkMaksi}
Let $M$ be an oriented maximal surface in the Minkowski space $\mathbb E^4_1$. 
If $M$ has pointwise 1-type Gauss map of the first kind, then the function $f=\|h\|^2$ 
satisfies
\begin{eqnarray}
\label{MinkMaksYuz1CesK1}e_1(f)&=&-4\varepsilon\omega_{12}(e_2)f,\\
\label{MinkMaksYuz1CesK2}e_2(f)&=&4\varepsilon\omega_{12}(e_1)f,
\end{eqnarray}
where $\{e_1,e_2\}$ is a local orthonormal frame for tangent bundle of $M$
and $\varepsilon\in\{-1,1\}$.
\end{Lemma}

\begin{proof}
Let $M$ be a maximal surface in $\mathbb E^4_1$ with  pointwise 1-type Gauss map of the first kind. 
Then  Theorem \ref{MinMinimal1Cesit1tip} implies that $M$ has flat normal bundle. 
Thus, the shape operators can be  diagonalized simultaneously, i.e., 
there exists an orthornormal frame field $\{e_1,e_2,e_3,e_4\}$ on $M$ such that 
$
A_\beta  =   \mathrm{diag}(h^\beta_{11},-h^\beta_{11}),$ $  \beta=3,4, 
$
as $H=0$. Therefore,  we have from  \eqref{Mink2esFormUzTanim} that
\begin{eqnarray}
\label{MinkMinGerekKos1ffonk} f=\|h\|^2&=&2\Big(\varepsilon_3(h^3_{11})^2+\varepsilon_4(h^4_{11})^2\Big).
\end{eqnarray}
and Codazzi equation  \eqref{MinkCodazziversion2} yields
\begin{eqnarray}
\label{MinkMinGerekKos1D1}e_1(h^3_{11})-\varepsilon_4 h^4_{11}\omega_{34}(e_1)&=&-2\omega_{12}(e_2)h^3_{11},\\
\label{MinkMinGerekKos1D2}e_1(h^4_{11})+\varepsilon_3 h^3_{11}\omega_{34}(e_1)&=&-2\omega_{12}(e_2)h^4_{11},\\
\label{MinkMinGerekKos1D3}e_2(h^3_{11})-\varepsilon_4 h^4_{11}\omega_{34}(e_2)&=&2\omega_{12}(e_1)h^3_{11},\\
\label{MinkMinGerekKos1D4}e_2(h^4_{11})+\varepsilon_3 h^3_{11}\omega_{34}(e_2)&=&2\omega_{12}(e_1)h^4_{11}.
\end{eqnarray}
By multiplying \eqref{MinkMinGerekKos1D1} and  \eqref{MinkMinGerekKos1D2}, 
respectively, $\varepsilon_3 h^3_{11}$ and  $\varepsilon_4 h^4_{11}$ 
and adding them, we obtain that
$$
\varepsilon_3h^3_{11}e_1(h^3_{11})+\varepsilon_4h^4_{11}e_1(h^4_{11})=-2\omega_{12}(e_2)\left( \varepsilon_3(h^3_{11})^2+\varepsilon_4(h^4_{11})^2\right).
$$
By using \eqref{MinkMinGerekKos1ffonk} again in this equation, 
we obtain  \eqref{MinkMaksYuz1CesK1}. In a similar way, we see that \eqref{MinkMinGerekKos1D3} 
and \eqref{MinkMinGerekKos1D4} give \eqref{MinkMaksYuz1CesK2}. 
\end{proof}


\begin{Prop} \label{PropMaxGlobalGauss}
Let $M$ be an oriented maximal surface  in the Minkowski space $\mathbb E^4_1$. 
Then $M$ has (global) 1-type Gauss map of the first kind if and only if the 
Gauss map $\nu$ of  $M$ is harmonic.
\end{Prop}
\begin{proof}
We assume that $M$ has (global) 1-type Gauss map $\nu$ of the first kind. 
Then Theorem \ref{MinMinimal1Cesit1tip} implies that $M$ has flat normal bundle.   
On the other hand, since $\nu$ is (global) 1-type of the first kind, 
\eqref{PW1TypeDefinition} is satisfied for $f=f_0$, where $f_0$ is a constant. 
Moreover, Lemma \ref{Lemma1MinkMaksi} implies that $f$ satisfies 
\eqref{MinkMaksYuz1CesK1} and \eqref{MinkMaksYuz1CesK2} from which we obtain
$\omega_{12}(e_1)f_0=\omega_{12}(e_2)f_0=0$
that imply $f_0=0$ or $\omega_{12} = 0$. In the case  $f_0=0$, 
we have $\Delta\nu=f_0\nu=0$, i.e., $\nu$ is harmonic.
Otherwise $M$ is flat, and  it follows from Proposition \ref{PropMax3kosuldenk} 
that $\nu$ is harmonic.

The converse is obvious.
\end{proof}

Now we study non-maximal space-like  surfaces in $\mathbb E^4_1$ 
with pointwise 1-type Gauss map of the first kind.

\begin{theorem}\label{NonmaksClassTheo}
 Let  $M$ be an oriented non-maximal space-like surface in $\mathbb E^4_1$. 
 Then  $M$ has pointwise 1-type Gauss map of the first kind if and only if $M$
  has parallel mean curvature vector. 
\end{theorem}

\begin{proof}
 Let  $M$ be an oriented non-maximal space-like surface in $\mathbb E^4_1$. 
 Suppose that  $M$ has pointwise 1-type 
 Gauss map of the first kind. Then \eqref{PW1TypeDefinition} is satisfied for 
 $f=\|h\|^2$ and $C=0$. From  \eqref{PW1TypeDefinition} and \eqref{MinkGaussLaplEnSon}
 we obtain that $R^D=0$ and
\begin{eqnarray}\label{ThmNonmaksClassTheoAraDenk1}
\nabla(\mathrm{tr}A_3)\wedge e_4+e_3\wedge \nabla(\mathrm{tr}A_4) + 
2\sum\limits_{j=1}^2\omega_{34} (e_j) H\wedge e_j=0.
\end{eqnarray} 
Since $R^D=0$, there exists a local orthonormal frame $\{e_3,e_4\}$ of normal bundle of $M$ 
such that $\omega_{34}= 0$. So, it follows  from \eqref{ThmNonmaksClassTheoAraDenk1} 
that  $\nabla\mathrm{tr}A_3=\nabla\mathrm{tr}A_4=0$, that is, 
$\mathrm{tr}A_{\beta} =constant, \; \beta = 3, 4,$ from which and $\omega_{34}=0$ we have  $DH=0$. 

Conversely, let $H$ be parallel. From Lemma \ref{MinkYuzOrtEgrVektPariseNormKnDuz}  
we have $R^D=0$. Thus, there exists  a local, orthonormal frame $\{e_3, e_4\}$ of normal 
bundle of $M$ such that $\omega_{34}\equiv 0$. So, it follows from $DH = 0$ that 
$\mathrm{tr}A_3$ and $\mathrm{tr}A_4$ are constants.
Therefore,  equation   \eqref{MinkGaussLaplEnSon} implies that  $\Delta\nu=\|h\|^2\nu$, 
 that is, $M$ has pointwise 1-type Gauss map of the first kind. 
\end{proof}

\begin{Example} \label{ChenVeken2009HoustonExample5.2}
Let $M$ be a surface in $\mathbb E^4_1$ given by 
\begin{align} \notag 
x(u,v)=\frac 1{ \sqrt2}&(u\cosh \sqrt2v, u\sinh \sqrt2v, \\ \notag 
 &\sqrt2\sin\sqrt2u-u\cos\sqrt2u,\sqrt2\cos\sqrt2u+u\sin\sqrt2u ).
 \end{align}
Then the mean curvature vector $H$ of $M$ is parallel  and light-like \cite{ChenVeken2009Houston}.
 Moreover, the Gaussian curvature of $M$ is $K=u^{-4}$ which implies $\|h\|^2=-2u^{-4}$ 
 from \eqref{Gauss-curvature}. Therefore, $M$ has proper 1-type Gauss map of the first 
 kind because of Theorem \ref{NonmaksClassTheo}, that is,  \eqref{PW1TypeDefinition} is
 satisfied for $C=0$ and $f=\|h\|^2=-2u^{-4}$.
\end{Example}

In \cite{ChenVeken2009Tohoku}, 
a complete classification of space-like surfaces with parallel mean curvature vector was given.
By combining \cite[Theorem 3.1]{ChenVeken2009Tohoku} and Theorem \ref{NonmaksClassTheo}, we have

\begin{theorem}\label{NonmaksClassTheo2ndversion}
Let $M$ be an oriented non-maximal space-like surface in $\mathbb E^4_1$ 
with space-like or time-like mean curvature vector. 
Then $M$ has pointwise 1-type Gauss map of the first kind if and only if  $M$ is 
a CMC surface lying in the light cone $\mathcal{LC}\subset\mathbb E^4_1$, 
a Euclidean hyperplane $\mathbb E^3\subset\mathbb E^4_1$, a Lorentzian hyperplane 
$\mathbb E^3_1\subset\mathbb E^4_1$, the de Sitter space-time 
$\mathbb S^3_1(c^2)\subset\mathbb E^4_1$, or the hyperbolic space $\mathbb H^3 (-c^2) \subset\mathbb E^4_1$.
\end{theorem}

In the next proposition we obtain characterization of non-maximal surfaces with (global) 1-type Gauss map of the first kind.
\begin{Prop} \label{PropMargTrapGlobalGauss}
Let $M$ be an oriented  space-like surface  in the Minkowski space $\mathbb E^4_1$ with light-like mean 
curvature vector. Then $M$ has (global) 1-type Gauss map of the first kind if and only 
if the Gauss map $\nu$ of $M$ is harmonic.
\end{Prop}
\begin{proof}
Suppose that  $M$ has non-harmonic, (global) 1-type Gauss map of the first kind with 
light-like mean curvature vector. Then \eqref{PW1TypeDefinition} is satisfied for 
$f=f_0$ and $C=0$,  where $f_0\neq0$  is a constant.  Moreover, Theorem \ref{NonmaksClassTheo} 
implies the mean curvature vector $H$ of $M$ is parallel. Since $\nu$ is non-harmonic, it follows from
Theorem \ref{MinkNonMaksHarmonikTheo} and  Theorem \ref{TheoremChenVeken2009Houston} that  
$M$ is congruent to a non-flat  surface lying in either $\mathbb S^3_1(r^2)$ or 
$\mathbb H^3(-r^2)$, and if $M$ is lying in $\mathbb S^3_1(r^2)$ (resp., in $\mathbb H^3(-r^2)$), 
then its mean curvature vector  $H'$ in $\mathbb S^3_1(c^2)$ (resp., in $\mathbb H^3(-r^2)$) 
satisfies $\langle H',H'\rangle=-r^2$ (resp., $\langle H',H'\rangle=r^2$). 
Also,  from Remark \ref{MinRemark1} we have $\|h^2\|=f_0$. 
 
Let $x$ be the position vector of $M$ in $\mathbb E^4_1$ and  
$\langle x, x\rangle=\varepsilon_3 r^{-2}$, where $\varepsilon_3=\pm 1$. 
We choose a local orthonormal frame $\{e_3,e_4\}$ of the normal bundle of $M$ such that 
$e_3=rx$  and $H=\varepsilon_3 r^2 (e_3-e_4)$. Since $e_3=rx$ is parallel, we have $\omega_{34}=0$. 
Moreover, $R^D=0$, i.e., $M$ has flat normal bundle. 
Thus, the shape operators of $M$ are  simultaneously diagonalizable. 
So, there exists a local orthonormal frame $\{e_1,e_2\}$ of tangent bundle of $M$ such that 
$ A_{\beta}=\mathrm{diag}(h^\beta_{11}, h^\beta_{22}), \beta =3, 4,$ 
and 
\begin{eqnarray}\label{PropMargTrapGlobalGaussDenk1} 
h^3_{11}+h^3_{22}=h^4_{11}+h^4_{22}=2r^2.
\end{eqnarray}
On the other hand, since $\|h\|^2$ and $\langle H,H\rangle$ are  constants,   
\eqref{Gauss-curvature} implies the Gaussian curvature $K$ of $M$ is constant, i.e., 
we have $K_0=\varepsilon_3 (h^3_{11}h^3_{22}-h^4_{11}h^4_{22})$,
 where $K_0\neq 0$ is a constant, from which and  \eqref{PropMargTrapGlobalGaussDenk1} we obtain that  
$$
e_i(h^3_{11})h^3_{22}+h^3_{11}e_i(h^3_{22})=e_i(h^4_{11})h^4_{22}+h^4_{11}e_i(h^4_{22})
$$ and 
\begin{equation}\label{an-equation-h}
e_i(h^\beta_{11})=-e_i(h^\beta_{22}),\quad i=1,2,\; \beta=3,4
\end{equation}
Using these equations we obtain that 
\begin{eqnarray}
\label{PropMargTrapGlobalGaussDenk2} e_i(h^3_{11})(h^3_{11}-h^3_{22})=e_i(h^4_{11})(h^4_{11}-h^4_{22}).
\end{eqnarray}
In addition, considering \eqref{an-equation-h}, the Codazzi equation \eqref{MinkCodazzi} yields
\begin{eqnarray}
\label{PropMargTrapGlobalGaussCodazzi1} e_1(h^\beta_{11})
&=&-\omega_{12}(e_2)(h^\beta_{11}-h^\beta_{22}),\\
\label{PropMargTrapGlobalGaussCodazzi2} e_2(h^\beta_{11})
&=&\omega_{12}(e_1)(h^\beta_{11}-h^\beta_{22}), \quad \beta=3,4.
\end{eqnarray}
So, it follows from these equations and   \eqref{PropMargTrapGlobalGaussDenk2} that  
\begin{eqnarray} \label{PropMargTrapGlobalGaussDenk3}
\omega_{12}(e_i)\Big((h^3_{11}-h^3_{22})^2-(h^4_{11}-h^4_{22})^2\Big)=0,\quad i=1,2.
\end{eqnarray}
As  $M$ is not flat, we have $\omega_{12}\not = 0$. Thus,  \eqref{PropMargTrapGlobalGaussDenk3}  
implies $(h^3_{11}-h^3_{22})^2=(h^4_{11}-h^4_{22})^2$ from which and \eqref{Gauss-curvature}  
we get  $f_0=\|h\|^2=0$ which is a contradiction.  Therefore, the Gauss map $\nu$ is harmonic.

The converse is obvious.
\end{proof}

Next we give a characterization for non-maximal space-like surfaces in the Minkowski space $\mathbb E^4_1$ 
with (global) 1-type Gauss map of the first kind.

\begin{theorem}\label{MinkNonmaksimGlob1cesTeorem}
Let $M$ be an oriented non-maximal surface  in the Minkowski space $\mathbb E^4_1$. 
Then $M$ has (global) 1-type Gauss map of the first kind if and only if 
$M$ has parallel mean curvature vector and constant Gaussian curvature.
\end{theorem}
\begin{proof}  
Let $M$ be an oriented non-maximal surface  in Minkowski space $\mathbb E^4_1$.
First we assume that $M$ has (global) 1-type Gauss map of the first kind. 
Then it follows from  \eqref{PW1TypeDefinition} and \eqref{MinkGaussLaplEnSon} that 
$\|h\|^2=f_0$ for some constant $f_0$.  
Also, Theorem \ref{NonmaksClassTheo} implies that  $M$ has parallel 
mean curvature vector which implies $\langle H,H\rangle$ is constant.
Therefore, \eqref{Gauss-curvature} implies that the Gaussian curvature $K$ of $M$ is constant.

Conversely, let $M$ has parallel mean curvature vector and constant Gaussian curvature.
By Theorem \ref{NonmaksClassTheo} we have $\Delta \nu = \|h\|^2 \nu$. 
Also, equation \eqref{Gauss-curvature} implies that $\|h\|^2$ is constant. 
Therefore the Gauss map of $M$ is of 1-type of the first kind.
\end{proof}

Next we give an example of a surface with  non-harmonic (global) 1-type Gauss map 
of the first kind:
\begin{Example}
Let $M$ be a surface in $\mathbb E^4_1$ given by 
$$x(u,v)=(a \cosh u,a \sinh u,b \cos v,b\sin v ), \quad b^2-a^2\neq 0,\quad ab\neq 0.$$ 
Let $c= \sqrt{|b^2-a^2|}.$ 
Then we have $M=H^1(-a^{-1})\times S^1(b^{-1})\subset S^3_1(c^{-2})\subset \mathbb E^4_1$ 
if $b^2-a^2>0$, and  $M=H^1(-a^{-1})\times S^1(b^{-1})\subset H^3(-c^{-2})\subset \mathbb E^4_1\;$ 
if $b^2-a^2<0$. 
By a direct calculation, it can be  seen that  $M$ has parallel mean curvature vector 
and constant Gaussian curvature. 
Hence, Theorem \ref{MinkNonmaksimGlob1cesTeorem} implies $M$ has (global) 1-type Gauss map of the first kind.
\end{Example}


\section*{Acknowledgements}
This work which is a part of the second author's doctoral thesis is partially supported by Istanbul Technical University.



\end{document}